\documentclass[12pt,reqno]{amsart}

\usepackage{amssymb}
\usepackage{hyperref}
\usepackage{tikz, tkz-tab, tikz-cd}
\usepackage[new]{old-arrows}

\newtheorem{theorem}{Theorem}[section]
\newtheorem{prop}[theorem]{Proposition}
\newtheorem{lemma}[theorem]{Lemma}

\newtheorem{question}[theorem]{Question}

\newtheorem{conj}[theorem]{Conjecture}

\newtheorem{mainthm}{Theorem}

\theoremstyle{definition}
\newtheorem{definition}[theorem]{Definition}
\newtheorem{remark}[theorem]{Remark}

\numberwithin{equation}{section}

\newcommand{\C}{\mathbb{C}}

\newcommand{\N}{\mathbb{N}}
\newcommand{\Q}{\mathbb{Q}}

\newcommand{\Z}{\mathbb{Z}}

\newcommand{\OO}{\mathcal{O}}
\newcommand{\PP}{\mathbb{P}}

\newcommand{\wt}{\widetilde}

\newcommand{\Hom}{\operatorname{Hom}}

\newcommand{\rank}{\operatorname{rank}}

\usetikzlibrary{decorations.markings}
\tikzstyle{vertex}=[circle, draw, inner sep=0pt, minimum size=5pt]

\usepackage[T1]{fontenc}
\usepackage{mathabx}

\usepackage[theoremfont]{newtxtext}

\begin{document}

\title[On compact complex surfaces with finite homotopy rank sum]{On compact complex surfaces with\\finite 
homotopy rank-sum}

\author[I. Biswas]{Indranil Biswas}

\address{Department of Mathematics, Shiv Nadar University, NH91, Tehsil
Dadri, Greater Noida, Uttar Pradesh 201314, India}

\email{indranil.biswas@snu.edu.in, indranil29@gmail.com}

\author[B. Hajra]{Buddhadev Hajra}

\address{School of Mathematics, Tata Institute of Fundamental
Research, Homi Bhabha Road, Mumbai 400005, India}

\email{hajrabuddhadev92@gmail.com}



\begin{abstract}
A topological space (not necessarily simply connected) is said to have \emph{finite homotopy rank-sum} if the sum of the ranks of all higher homotopy groups (from the second homotopy group onward) is finite. In this article, we characterize the smooth compact complex K{\"a}hler surfaces having finite homotopy rank-sum. We also prove the Steinness of the universal cover of these surfaces assuming holomorphic convexity of the universal cover.
\end{abstract}

\maketitle

\tableofcontents

\section{Introduction}\label{se1}

The higher homotopy groups of any topological space $X$ are always abelian, and for $i\,\geq\,2$, $\rank(\pi_i(X))$ is defined to be 
the dimension of the $\Q$-vector space $\pi_i(X)\otimes_\Z\Q$, which need not be finite for a finite CW complex. For example, 
$\rank(\pi_2(X))$ of the wedge sum $S^1\vee S^2$ is infinite as its universal cover is the real line with a copy of $S^2$ attached to every integer. However, J.-P. Serre proved that for a simply 
connected finite CW complex $X$, all the higher homotopy groups are finitely generated. More generally, if $X$ is a finite CW complex with $\pi_1(X)$ a finite group, all the higher homotopy groups of $X$ are finitely generated. 

For any path connected topological space $X$, the graded $\Q$-vector spaces $$\bigoplus_{i\,\geq\, 2}\pi_i(X)\otimes_\Z\Q \ \ \text{ and } \ \ \bigoplus_{i\,\geq\, 0} H^i(X;\,\Q)$$ are denoted by $\pi_\ast(X;\,\Q)$ and $H^\ast(X;\,\Q)$ respectively.

A simply connected topological space $X$ is of \emph{rationally elliptic homotopy type} (cf. \cite[Part VI, \S~32]{FHT}) if it satisfies the following two conditions:
\begin{align}
 &\dim (\pi_\ast(X;\, {\mathbb Q}))
\,:=\,\sum\limits_{i\,\geq\,2}\rank(\pi_i(X))\,=\,\sum\limits_{i\,\geq\,2}\dim \pi_i(X)\otimes_\Z \Q\,<\, \infty
,\label{Defn: Finite homotopy rank-sum}\\
\text{and}&\nonumber\\
&\dim(H^\ast(X;\,\Q))\,:=\,\sum\limits_{i\,\geq\, 0}\dim(H^i(X;\,\Q))\,=\,\sum\limits_{i\,\geq\,0}b_i(X)\,<\,\infty.
\end{align}

A topological space $X$ (not necessarily simply connected) is said to have \emph{finite homotopy rank-sum} or is said to \emph{satisfy the finite 
homotopy rank-sum property} if $\sum_{i\,\geq\,2}\dim \pi_i(X)\otimes_\Z \Q \,<\, \infty$ (see \eqref{Defn: Finite homotopy rank-sum}).

An example of elliptic homotopy type topological space is the real $n$-sphere $S^n$ with $n\, \geq\,1$. Any complex rational homogeneous space is of elliptic homotopy type. Simply connected compact K{\"a}hler manifolds of the elliptic homotopy type of complex dimension up to $3$ were described in \cite{Amo-Bis}.

However, once the simply connectedness condition is dropped, characterizing the compact complex surfaces with finite homotopy rank-sum becomes 
significantly more challenging compared to the results presented in \cite[Theorem 1.1]{Amo-Bis}. Our primary aim in this article is 
to provide a characterization of compact complex (K{\"a}hler) surfaces possessing the finite homotopy rank-sum property.

Similarly, a recent work by the authors in \cite{BH} has characterized complex Stein surfaces exhibiting an elliptic homotopy type.

\subsection{Main results}

We first prove the following result which generalizes \cite[Theorem 1.1]{Amo-Bis}. 

\begin{mainthm}[{Theorem \ref{Thm: Characterization of quasi-elliptic compact surfaces with finite fundamental group}}]\label{Thm A}
Let $X$ be a smooth compact complex surface having a finite fundamental group satisfying the finite homotopy rank-sum property. Then $X$ is in fact simply connected and of the rationally elliptic homotopy type.

Consequently, $X$ is one of the surfaces listed below:
\begin{enumerate}
\item[(a)] the complex projective plane $\C\PP^2$;

\item[(b)] Hirzebruch surfaces $\mathbb{S}_h\,=\,\PP(\mathcal{O}_{\C\PP^1} \oplus \mathcal{O}_{\C\PP^1}(h))$, where
$h\,\geq\,0$ is an integer; and
\item[(c)] Hirzebruch's fake quadrics (see Remark \ref{re1}), if they exist.
\end{enumerate}
\end{mainthm}

\begin{remark}\label{re1}
Let $X$ be a projective surface of the general type with $$q(X)\,=\,p_g(X)\,=\,0,\,\, c_1^2=K_X^2=8,\,\,\text{and}\,\,c_2(X)
\,=\,e(X)\,=\,4.$$
Such a surface $X$ is very often referred to as \emph{fake quadrics}. Fake quadrics do exist. In fact, all
fake quadrics whose bi-canonical map $\Phi_{|2K_X|}$ is of degree $2$ are classified
by M. Mendes Lopes and R. Pardini in \cite{Lop-Par}; the fundamental group of each of these fake quadrics
is non-trivial. Many fake quadrics with bi-canonical maps of degree $1$ were found by Bauer, Catanese,
Grunewald, and Pignatelli \cite{B-C-G-P}. They all have an infinite fundamental
group and their universal cover is the bi-disk $\mathbb{H}\times \mathbb{H}$. Hirzebruch asked if a simply-connected
fake quadric exists. This question is still open. By Freedman's theorem, any simply connected fake quadric is either
homeomorphic to the Hirzebruch surface $\mathbb{S}_1$ or homeomorphic to the (minimal) quadric $\mathbb{S}_0\,=\,\C\PP^1\times \C\PP^1$.
\end{remark}

Next, we prove the following result in the case of smooth compact complex K{\"a}hler surfaces having infinite fundamental groups and satisfying the finite homotopy rank-sum property.

\begin{mainthm}[Theorem \ref{Thm: Characterization of quasi-elliptic compact surfaces with infinite fundamental group2}]\label{Thm B}
Let $X$ be a smooth compact complex K{\"a}hler surface having an infinite fundamental group whose universal cover $\wt{X}$ is
holomorphically convex. Then $X$ satisfies the finite homotopy rank-sum property if and only if one of the following statements hold:
\begin{enumerate}
\item $X$ is an Eilenberg-MacLane $K(\pi,\,1)$-space with $\pi_1(X)$ being a $\rm{PD(4)}$-group;

\item $\wt{X}$ is homotopic to $S^2$, the real $2$-sphere.
\end{enumerate}
Consequently, $\wt{X}$ is a Stein manifold.
\end{mainthm}

The classification of smooth complex compact surfaces which are $K(\pi,\,1)$-spaces is known due to S. R. Gurjar--P. Pokale (cf. \cite{GP}). Therefore, our next result characterizes the smooth complex projective surfaces which are uniformized by a smooth complex analytic manifold (of complex dimension $2$) homotopic to $S^2$. This gives a full characterization for the non-general type smooth projective surfaces. However, we characterized such smooth projective surfaces of general type based on certain assumptions. The result says the following:

\begin{mainthm}[Theorem \ref{Thm: Characterization of compact surfaces having infinite fundamental group and sphere as universal cover}]\label{Thm C}
Let $X$ be a smooth complex projective surface having an infinite fundamental group. Then the following
statements hold:
\begin{enumerate}
\item If $\kappa(X) \,\leq\, 1$ and $\rank \pi_2(X)\,=\,1$, then $X$ is one of the following two types:
\begin{enumerate}
\item $\kappa(X)\,=\,-\,\infty$ and $X$ is a non-rational ruled surface;

\item $\kappa(X)\,=\,1$ and $X$ admits an elliptic fibration $f\,:\,X\,\longrightarrow\, D$ with $D\,\cong\,
\PP^1$, whose every fiber is an elliptic curve with reduced structure, such that $f$ has at most $3$ multiple fibers satisfying the additional condition that
if the number of singular fibers is three, then their multiplicities form a platonic triple.
\end{enumerate}
In both of the above cases, the universal cover $\widetilde X$ is homotopic to $S^2$.

\item If $\kappa(X)\,=\,2$ with $\wt{X}$ being homotopic to $S^2$, the following two hold:
\begin{enumerate}
\item If $G\,:=\,\pi_1(X)$ and $H^4(G;\,\Q)\,=\,H^5(G;\,\Q)=0$, then 
\[
H^i(G;\,\Q)\,=\,
\begin{cases}
\Q & \text{if } i\,=\,0 \,\text{ or }\, 2,\\
0 & \text{otherwise},
\end{cases}
\]
and $X$ must be a fake quadric with $\wt{X}$ homotopic to $S^2$, if it exists\footnote{All fake quadric surfaces known
so far are uniformized by some complex contractible manifold. Therefore all these known examples are in
fact examples of $K(\pi,\,1)$-surfaces.}.
\item $\pi_1(X)$ is non-abelian.
\end{enumerate}
\end{enumerate}
\end{mainthm}

The following is a consequence of Theorem \ref{Thm C}.

\begin{mainthm}[Proposition \ref{Prop: No compact surface of general type having fundamental group of a compact curve and sphere as universal cover}]\label{Thm D}
Let $X$ be a smooth projective surface such that its universal cover is homotopic to $S^2$. Then the
following two statements hold:
\begin{enumerate}
\item If $X$ contains a smooth rational curve, then $X$ must be a non-rational ruled surface.

\item If $G\, :=\,\pi_1(X)$ is a surface group, then $X$ is either a non-rational ruled surface, or
it is an elliptic surface admitting a relatively minimal elliptic fibration
$\varphi\,:\, X \,\longrightarrow\, D$, with $D\,\cong \, \PP^1$, such that $\chi(X,\,\mathcal{O}_X)\,=\,0$,
and $\varphi$ has at most three multiple fibers with multiplicities forming a platonic triple whenever there are exactly three multiple fibers.
\end{enumerate}
\end{mainthm}

Carlson and Toledo conjectured the following. We refer the reader \cite{Kol} for more details about this conjecture.

\begin{conj}[Carlson--Toledo Conjecture]
Let $G$ be an infinite K{\"a}hler group, i.e., the fundamental group of a smooth complex K{\"a}hler manifold. Then virtually the second Betti number $b_2(G)$ of $G$ is positive, i.e., there exists a finite index subgroup $G^\circ$ of $G$ such that the second Betti number $b_2(G^\circ)$ of $G^\circ$ is positive.
\end{conj}

Based on this conjecture, Domingo Toledo asked the following very interesting question (see, \cite[Question 1.6.]{BMP}).

\begin{question}
Let $G$ be a K{\"a}hler group such that $H^2(G,\,\Z G)\,\neq\,0$. Is $G$ the fundamental group of a compact Riemann surface?
\end{question}
In connection with the above question, we prove the following result for which we invoke the notion of \emph{homotopical height} of a group
for a subclass $\mathcal{C}$ of the class of manifolds, denoted by ${\rm ht}_{\mathcal{C}}(G)$ (the formal definition is given in
a later section). 

\subsection{Notation}

We reserve the following notation:

`$\mathcal{S}$' denotes the class of smooth complex projective manifolds whose universal cover is a Stein manifold.

\begin{mainthm}[Theorem \ref{Thm: Homotopical height of surfaces uniformized by 2-sphere}]
Let $X$ be a smooth projective surface of general type such that its universal cover is a Stein manifold
which is homotopic to $S^2$. Let $G\,:=\,\pi_1(X)$. Then the following statements hold:
\begin{enumerate}
\item If ${\rm ht}_{\mathcal{S}}(G)\,>\,2$, then $H^2(G,\,\Z G)\,=\,0$.
\item If ${\rm ht}_{\mathcal{S}}(G)\,=\,2$ with $M$ being a smooth complex projective surface realizing
${\rm ht}_{\mathcal{S}}(G)$, and the action of $\pi_1(M)$ on $\pi_2(M)$ is the trivial one, then $H^2(G,\,\Z G)\,=\,0$.
\end{enumerate}
\end{mainthm}

\section{Proof of the main results}

\begin{theorem}\label{Thm: Characterization of quasi-elliptic compact surfaces with finite fundamental group}
Let $X$ be a smooth compact complex surface having a finite fundamental group satisfying the finite homotopy rank-sum
property. Then $X$ is in fact simply connected and of the rationally elliptic homotopy type.

Consequently, $X$ is one of the surfaces listed below:
\begin{enumerate}
\item[(a)] the complex projective plane $\C\PP^2$,

\item[(b)] Hirzebruch surfaces $\mathbb{S}_h\,:=\,\PP
(\mathcal{O}_{\C\PP^1}\oplus \mathcal{O}_{\C\PP^1}(h))$, where $h\,\geq\,0$ is an integer, and

\item[(c)] simply connected fake quadrics, if they exist.
\end{enumerate}
\end{theorem}

\begin{proof} Let $p\,: \,\wt{X}\,\longrightarrow \, X$ be the universal covering. Since $\pi_1(X)$ is finite, $\wt{X}$ still remains a compact complex surface. Therefore $H^*(\wt{X};\,\Q)$ is finite dimensional. Additionally, since $X$ enjoys the finite homotopy rank-sum property, it follows that $\wt{X}$ is a rationally elliptic $1$-connected compact complex (K{\"a}hler) surface. Then the complete 
classification of $\wt{X}$ is as mentioned in \cite[Theorem 1.1(a)--(c)]{Amo-Bis}. Clearly, $e(\PP^2)\,=\,3$. Again, 
Hirzebruch surfaces are topologically locally trivial $\PP^1$-fiber bundle over $\PP^1$. Therefore $e(\mathbb{S}_h)\,=\,e(\PP^1)\,\cdot\, 
e(\PP^1)\,=\,4$ for every $h\,\geq\,0$. Next, the Euler characteristic of every fake quadric coincides 
with that of the quadric surface $\PP^1\times \PP^1$, which is $4$.

Now, as $X$ is a compact complex surface, $e(X)\,=\,2\,-\,2b_1(X)\,+\,b_2(X)$. This implies that
\begin{equation}\label{f1}
e(X)\,=\,2\,+\,b_2(X),
\end{equation}
because $\pi_1(X)$ is finite. Since $\wt{X}$ is K{\"a}hler, and $p$ is a finite covering, it implies that
$X$ is also K{\"a}hler. Consequently, $b_2(X)\,\geq\, 
h^{1,1}(X)\,\geq\,1$. Therefore, it follows from \eqref{f1} that $e(X)\,\geq\, 3$.

We observed above that $e(\wt{X})$ is either $3$ or $4$. Since $e(\wt{X})\,=\,\deg p\,\cdot\, e(X)$, it turns out that $\deg p\,=\,1$,
whence $\pi_1(X)$ is trivial. The proof is now completed using the classification in \cite[Theorem 1.1(a)--(c)]{Amo-Bis}.
\end{proof}

Although, the definition of finite homotopy rank-sum property involves the information about the ranks of all the higher homotopy groups of a 
topological space, we will see that for compact complex surfaces (having infinite fundamental groups) this property is determined by 
the second homotopy groups only.

\begin{remark}\label{Rem: pi_2 of a 4-manifold is torsion-free}
It is important to note that $\pi_2$ of a smooth complex projective surface is always torsion-free. In \cite{Gur2}, R. V. Gurjar proved this for such surfaces if the universal cover happens to be holomorphically convex. In fact, $\pi_2$ becomes free for these surfaces. Soon after, in 2004, Jerome P. Levine and Daniel Ruberman observed the following more general result after some correspondence with R. V. Gurjar: ``\emph{The second homotopy group of a compact, connected $4$-manifold is torsion-free}''. But it seems to us that Levine-Ruberman never published this result\footnote{We refer to the answer by Ruberman to the MathOverflow question under the following link:\\
\url{https://mathoverflow.net/questions/137757}. See also \cite[Proposition 3.1]{GP}.}.
\end{remark}

Now we prove the following characterization result.

\begin{theorem}\label{Thm: Characterization of quasi-elliptic compact surfaces with infinite fundamental group1}
Let $X$ be a compact complex K{\"a}hler surface having an infinite fundamental group satisfying the finite homotopy rank-sum property. Then one of the following holds:
\begin{enumerate}
\item $X$ is an Eilenberg-MacLane $K(\pi,\,1)$-space with $\pi_1(X)$ a Poincar{\'e} duality $\rm{PD(4)}$-group;

\item $\rank \pi_2(X)\,=\,1$.
\end{enumerate}

Moreover, if the universal cover $\wt{X}$ of $X$ is holomorphically convex, then $\pi_2(X)$ has rank
$1$ if and only if $\wt{X}$ is homotopic to $S^2$.
\end{theorem}

\begin{proof}
Let $p\,:\, \wt{X}\,\longrightarrow \,X$ be the universal covering. Since $X$ satisfies finite homotopy rank-sum property, it implies that $\pi_i(X)\,=\,\pi_i(\wt{X})$ is of finite rank for all $i\,\geq\,2$. Hence by Hurewicz's theorem,
$H_2(\wt{X};\,\Q)$ is a finite-dimensional $\Q$-vector space. Again Hurewicz's theorem implies that the Hurewicz homomorphism
$$h_3\,:\,\pi_3(\wt{X})\longrightarrow\,H_3(\wt{X};\,\Z)$$ is surjective and therefore $h_3\,\otimes\,{\rm Id}
\,:\, \pi_3(\wt{X})\,\otimes_{\Z}\,\Q\longrightarrow\,H_3(\wt{X};\,\Q)$ is also surjective. Thus
$H_3(\wt{X};\,\Q)$ is also a
finite-dimensional $\Q$-vector space, and $\dim H_3(\wt{X};\,\Q)\,\leq\,\dim \pi_3(\wt{X})\,\otimes_{\Z}\,\Q$. Consequently, $H^i(X;\,\Q)$ is finite-dimensional using the universal coefficient theorem for $i\,=2,\,3$. Since $\pi_1(X)$ is infinite, $\wt{X}$ is a non-compact $4$-manifold and therefore $H^4(\wt{X};\,\Z)\,=\,0$. Thus it follows that the graded $\Q$-vector space $H^\ast(\wt{X};\,\Q)$ is finite-dimensional. Hence by definition, $\wt{X}$ is of rationally elliptic homotopy type. Therefore, 
using Friedlander--Halperin's result, \cite[Corollary 1.3(2)]{Fri-Hal}, it is deduced that
\begin{equation}\label{e1}
\sum_{k\,\geq \,1}2k\dim \pi_{2k}(\wt{X})\,\otimes_\Z\,\Q\,\,\leq\,\, 3,
\end{equation}
which yields that $\rank \pi_2(X)\,\leq\,1$. 

Recall that, $\pi_2(X)\,=\,\pi_2(\wt{X})$ is torsion-free (cf. Remark \ref{Rem: pi_2 of a 4-manifold is torsion-free}). Therefore if $\rank \pi_2(X)\,=\,0$, it follows that $\pi_2(X)$ is trivial. In 
\cite{Gro}, Gromov proved that every infinite K{\"a}hler group is a one-ended group. Hence from \cite[Lemma 2.1]{BMP} it follows that 
$X$ is an Eilenberg-MacLane $K(\pi,\,1)$-space with $\pi_1(X)$ a Poincar{\'e} duality $\rm{PD(4)}$-group. This completes the proof
of the first part.

For the last part, we assume that $\wt{X}$ is a holomorphically convex manifold. Then using \cite[Proposition 3.7]{GP}
we have $H_3(\wt{X};\,\Z)\,=\,0$. Since $\wt{X}$ is holomorphically convex, $\pi_2(X)$ is a free abelian group (cf. Remark \ref{Rem: pi_2 of a 4-manifold is torsion-free}). If $\rank \pi_2(X)\,=\,1$, it turns out that $\pi_2(X)\,=\,\pi_2(\wt{X})\,\cong\,\Z$. Therefore using Hurewicz's theorem, it follows that $H_2(\wt{X};\,\Z)\,\cong\,\Z$. Also note that, the $i$-th reduced homology group $\widetilde{H}_i(\wt{X};\,\Z)$ is trivial for all $i\,\neq\,2$. Consequently, $\wt{X}$ is a Moore $M(\Z,\,2)$-space, which implies that $\wt{X}$ is homotopically equivalent to the real $2$-sphere.
This completes the proof.
\end{proof}

A famous conjecture by Igor R. Shafarevich is as follows:

\begin{conj}[Shafarevich's Conjecture]\label{Conj: SC}
The universal cover of a smooth complex compact K{\"a}hler variety is holomorphically convex.
\end{conj}

\begin{remark}\label{Rem: Shafarevich Conj}
For smooth complex projective surfaces, the above conjecture is known to be true in many cases. For example, 
it is known to be true for all smooth complex projective surfaces of non-general type, i.e., of the Kodaira 
dimension $\leq\,1$. A nice summary of cases when this conjecture holds for smooth projective surfaces can be 
found in \cite{Gur1}, \cite{GuP}. Also, in the context of smooth K{\"a}hler varieties, in certain cases, the 
above conjecture is known to be true. Most often, there are conditions on the type of the fundamental group 
or the existence of some special kind of representation of the fundamental group.
\end{remark}

We prove the following. 

\begin{theorem}\label{Thm: Characterization of quasi-elliptic compact surfaces with infinite fundamental group2}
Let $X$ be a smooth compact complex K{\"a}hler surface having an infinite fundamental group whose universal cover $\wt{X}$ is holomorphically convex. Then $X$ satisfies the finite homotopy rank-sum property if and only if one of the following holds:
\begin{enumerate}
\item $X$ is an Eilenberg-MacLane $K(\pi,\,1)$-space with $\pi_1(X)$ a $\rm{PD(4)}$-group;

\item $\wt{X}$ is homotopic to $S^2$.
\end{enumerate}
Consequently, $\wt{X}$ is a Stein manifold.
\end{theorem}

\begin{proof}
The first part is obvious from Theorem \ref{Thm: Characterization of quasi-elliptic compact surfaces with infinite fundamental group1} (see also Remark \ref{Rem: Shafarevich Conj}). Therefore 
we will only prove here the Steinness of $\widetilde X$.

Let $M$ be any smooth complex compact surface that is an Eilenberg-MacLane $K(\pi,\,1)$-space. If $M$ has a holomorphically convex universal cover $\wt{M}$, then $\wt{M}$ is a Stein contractible manifold. We refer the reader to \cite[Theorem 6.2]{GGH} for this. 

So it remains to prove the Steinness of $\wt{X}$ whenever $\wt{X}$ is holomorphically convex and is homotopic to $S^2$. 

If possible, assume that $\wt{X}$ is not a Stein manifold. Then there is a Cartan--Remmert reduction 
$f\,:\,\wt{X}\,\longrightarrow\, Y$, which is a proper complex analytic map to a normal Stein space $Y$ with connected fibers. Since $f$ is 
proper, the fibers of $f$ are compact complex submanifolds of $\wt{X}$. Then we have the following two 
possibilities:
	
{\bf Case 1.}\,\, {\it When $\dim Y\,=\,2$.} 
 
In this case, the general fiber of $\varphi$ is discrete. Since fibers of $f$ are connected, the general fiber is just a singleton, 
i.e., $f$ is a bimeromorphism but it is not an isomorphism because $\wt{X}$ is not a Stein space (by assumption). Therefore there is a 
fiber of $f$ which contains a complete complex curve as one of its irreducible components. If we choose one such irreducible complete curve $C$ contained in a singular fiber of $f$, using a result of Grauert (cf. \cite{Gra}) it follows that $C^2\,<\,0$, because $C$ is
contracted to a point under the proper complex analytic map $f$.
 
Clearly, $C$ defines a complex analytic line bundle $\mathcal{O}_{\widetilde X}(C)$. Since $C^2\,<\,0$,
it follows that $\mathcal{O}_{\widetilde X}(C)$ is a non-trivial element 
in the analytic Picard group $H^1(\wt{X},\,\mathcal{O}_{\wt{X}}^*)$. The exponential sequence of complex analytic sheaves, $$0 \,\longrightarrow\, \Z \,\longrightarrow\, 
\mathcal{O}_{\wt{X}} \,\longrightarrow\, \mathcal{O}_{\wt{X}}^*\,\longrightarrow\, 0$$ induces a long exact sequence of sheaf 
cohomologies, viz.
\begin{equation}\label{eqn: sheaf cohomology sequence}
\cdots \,\longrightarrow\, H^1(\wt{X},\, \mathcal{O}_{\wt{X}}) \,\longrightarrow\, H^1(\wt{X},\,
\mathcal{O}_{\wt{X}}^*) \,\xrightarrow{\,\,c_1\,\,}\, H^2(\wt{X},\,\Z)\,\longrightarrow\, \cdots.
\end{equation}
If possible, let $$\mathcal{O}_{\widetilde X}(C)\,\,\in\,\,{\rm Im}\left(H^1(\wt{X},\, \mathcal{O}_{\wt{X}}) \,\longrightarrow\, H^1(\wt{X},\,
\mathcal{O}_{\wt{X}}^*)\right).$$ Thus the line bundle $\mathcal{O}_{\widetilde X}(C)|_{C}$ on $C$ lies in the image of the map
$H^1(C,\, \mathcal{O}_{C}) \,\longrightarrow\, H^1(C, \,\mathcal{O}_{C}^*)$. Hence $\deg(\mathcal{O}_{\widetilde X}(C)|_{C})\,=\,0$, which
implies $C^2\,=\,0$, a contradiction. Therefore $$\mathcal{O}_{\widetilde X}(C)\,\notin\,{\rm Im}\left(H^1(\wt{X},\,
\mathcal{O}_{\wt{X}}) \,\longrightarrow\, H^1(\wt{X}, \,\mathcal{O}_{\wt{X}}^*)\right),$$ and hence from the exactness of
\eqref{eqn: sheaf cohomology sequence} it follows that $c_1(\mathcal{O}_{\widetilde X}(C))$ is a non-trivial element in $H^2(\wt{X};\,\Z)$.
Since $\wt{X}$ is homotopic to $S^2$, we have $H^2(\wt{X};\,\Z)\,\cong\,\Z$. Assume that
$\alpha\,\in\,H^2(\wt{X};\,\Z)$ generates $H^2(\wt{X};\,\Z)$. So $c_1(\mathcal{O}_{\widetilde X}(C))\,=\,n\alpha$ for an integer $n\,\neq\,0$.
 
Since $\pi_1(X)$ is infinite, there are infinitely many complete curves $C_1,\,C_2,\,\ldots$ in $\wt{X}$ 
which are the $\pi_1(X)$-translates of $C$. Also note that, $C_i^2\,=\,C^2\,<\,0$. We will now prove that 
the homology classes $[C_1],\,[C_2],\,\ldots$ corresponding to $C_1,\,C_2,\,\ldots$ respectively are all independent in $H^2(\wt{X};\,\Z)$; here we 
apply the universal coefficient theorem for $H^2(\wt{X};\,\Z)\,\cong\,\Hom(H_2(\wt{X};\,\Z),\,\Z)$. Without 
loss of generality, assume that $C_1$ is a non-trivial $\pi_1(X)$-translate of $C$. Therefore $C$ and $C_1$ 
must lie in two different $\pi_1(X)$-orbits of $\wt{X}$. Hence $C_1\cap C\,=\,\emptyset$, i.e., the 
intersection number $C_1\cdot C\,=\,0$. Let $\mathcal{O}_{\widetilde X}(C_1)$ be the complex analytic line bundle on 
$\wt{X}$ corresponding to the complete curve $C_1$. Since $C_1^2\,<\,0$, by the same argument as above it 
follows that $c_1(\mathcal{O}_{\widetilde X}(C_1))$ is non-trivial in $H^2(\wt{X},\,\Z)$. Thus, 
$c_1(\mathcal{O}_{\widetilde X} (C_1))\,=\,m\alpha$ for an integer $m\,\neq\,0$. Therefore, it is evident 
that $n\cdot c_1(\mathcal{O}_{\widetilde X}(C_1))\,=\,m\cdot c_1(\mathcal{O}_{\widetilde 
X}(C))\,=\,mn\alpha$, and thus
\begin{equation}\label{eqn: degree}
 n\cdot \deg(\mathcal{O}_{\widetilde X}(C_1)|_{C})\,=\,m\cdot \deg(\mathcal{O}_{\widetilde X}(C)|_{C}).
\end{equation}
Note that $\deg(\mathcal{O}_{\widetilde X}(C_1)|_{C})\,=\,C_1\cdot C\,=\,0$ and $\deg(\mathcal{O}_{\widetilde X}(C)|_{C})
\,=\,C^2$. Thus from \eqref{eqn: degree} it follows that and $m\cdot C^2\,=\,0$, which in turn implies that $C^2\,=\,0$ as $m\,\neq\,0$. However, this is a contradiction which in fact proves that this case cannot occur at all.
	
{\bf Case 2.} {\it When $Y$ is a Riemann surface.}

In this case, $f$ is a proper morphism with irreducible general fiber. This situation appears in \cite[\S~2, Theorem 1]{Gur2} where
the author has proved that $b_2(\wt{X})$ is infinite. But in our case, this cannot happen either since $\wt{X}$ is homotopic to $S^2$. This completes the proof. 
\end{proof}

For a better understanding of Theorem \ref{Thm: Characterization of quasi-elliptic compact surfaces with infinite fundamental group1} and 
Theorem \ref{Thm: Characterization of quasi-elliptic compact surfaces with infinite fundamental group2} we mention the following 
classification result due to S. R. Gurjar--P. Pokale (cf. \cite[Lemma 8.3, Theorem 8.4]{GP}).

\begin{theorem}[{\cite{GP}}]
Let $X$ be a smooth compact connected complex $K(\pi,\,1)$ surface. Then the following statements hold:
\begin{enumerate}
\item $X$ is a minimal surface.

\item If $\kappa(X) \,=\, -\,\infty$, then $X$ is an Inoue surface.

\item If $\kappa(X) \,=\, 0$, then $X$ is either an Abelian surface, or a Hyperelliptic surface, or a Kodaira surface.

\item If $\kappa(X) \,=\, 1$, then $X$ is an elliptic surface with $e(X)\,=\,\chi(\OO_X)\,=\,0$.
\end{enumerate}
\end{theorem}

\begin{remark}
Note that, a complete characterization of smooth irreducible complex projective $K(\pi,\,1)$-surfaces of
general type (i.e., $\kappa\,=\,2$) is not yet achieved. But there are many examples of such surfaces. Any fake projective
plane $X$ satisfies the equality $c_1^2(X)\,=\,3c_2(X)$. From the differential geometric approach of proving the Bogomolov-Miyaoka-Yau inequality by S.-T. Yau (cf. \cite{Yau}) in fact proves that a smooth projective surface of general type satisfying $c_1^2\,=\,c_2$ is uniformized by a disk in $\C^2$.
Hence the universal cover of any fake projective
plane is contractible. Therefore, any fake projective plane is a
$K(\pi,\,1)$-surface. As we noted earlier in Remark \ref{re1} that some characterization is known about fake quadrics as well. There are examples of fake
quadrics that are uniformized by $\mathbb{H}\,\times\, \mathbb{H}$, where $\mathbb{H}$ denotes the upper half of the complex
plane. Therefore these fake quadrics are also examples of $K(\pi,\,1)$-spaces. We refer the reader to \cite{GGH}, especially the last section,
for more details about smooth $K(\pi,\,1)$ algebraic surfaces.
\end{remark}

So in our context, it remains to understand the characterization of the smooth compact complex K{\"a}hler surfaces $X$ whose universal 
cover $\wt{X}$ is homotopic to a real $2$-sphere. More generally, we will investigate the other left-out case when $\rank 
\pi_2(X)\,=\,1$, i.e., $\pi_2(X)\,\otimes_{\Z}\,\Q\,=\,\Q$. Evidently, the fundamental group of such a surface is infinite. 

The following lemma is useful for the further analysis.

\begin{lemma}\label{Lem: Minimality of sphere uniformized compact complex surfaces}
Let $X$ be a smooth compact complex surface such that $\pi_1(X)$ is infinite and $\rank\pi_2(X)$ is finite. Then $X$ is a minimal surface.
\end{lemma}

\begin{proof}
Suppose that $X$ is not minimal. Then there is a smooth rational curve $C \,\cong \, \PP^1$ with $C^2\,=\,-\,1$. Since $\pi_1(X)$ is
infinite, using the proof of \cite[Proposition 4.3]{GP} it turns out that $\rank(\pi_2(X))$ is infinite --- a contradiction. This
completes the proof.
\end{proof}

Our next result gives a concrete understanding of the smooth complex projective surfaces satisfying finite homotopy rank-sum property. Although many mathematicians have explored numerous features of surfaces of general type, there are still many mysteries about the nature of possible invariants
of these surfaces like the geometric genus 
$p_g$, irregularity $q$, Chern numbers $c_1^2$ and $c_2$, fundamental group etc. Therefore, in the case of general type smooth complex projective surfaces satisfying the finite homotopy rank-sum property, our characterization is based on certain vanishing assumptions on the group cohomologies of the fundamental group of those surfaces. The precise result is as follows:

\begin{theorem}\label{Thm: Characterization of compact surfaces having infinite fundamental group and sphere as universal cover}
Let $X$ be a smooth complex projective surface having an infinite fundamental group. Then the following
two hold:
\begin{enumerate}
\item If $\kappa(X) \,\leq\, 1$ and $\rank \pi_2(X)\,=\,1$, then $X$ is one of the following two types:
\begin{enumerate}
\item $\kappa(X)\,=\,-\,\infty$ and $X$ is a non-rational ruled surface;

\item $\kappa(X)\,=\,1$ and $X$ admits an elliptic fibration $f\,:\,X\,\longrightarrow\, D$ with $D\,\cong\,
\PP^1$, whose every fiber is an elliptic curve with reduced structure, such that $f$ has at most $3$ multiple fibers satisfying
the additional condition that if the number of singular fibers is three, then their multiplicities form a platonic triple.
\end{enumerate}
In both of the above cases, the universal cover $\widetilde X$ is in fact homotopic to $S^2$.

\item If $\kappa(X)\,=\,2$ with $\wt{X}$ being homotopic to $S^2$, the following statements hold:
\begin{enumerate}
\item If $G\,:=\,\pi_1(X)$ and $H^4(G;\,\Q)\,=\,H^5(G;\,\Q)\,=\,0$, then 
\[
H^i(G;\,\Q)\,=\,
\begin{cases}
        \Q & \,\text{if }\, i\,=\,0 \,\text{ or }\, 2;\\
        0 & \text{otherwise};
\end{cases}
\]
and $X$ must be a fake quadric with $\wt{X}$ homotopic to $S^2$, if exists\footnote{All fake quadric surfaces 
known so far are uniformized by complex contractible Riemannian manifolds. Therefore all these known examples 
are in fact examples of $K(\pi,\,1)$-surfaces.}.

\item $\pi_1(X)$ is non-abelian.
\end{enumerate}
\end{enumerate}
\end{theorem}

\begin{proof}
Let $p\,:\, \wt{X} \,\longrightarrow\, X$ be the universal covering. By using Hurewicz's theorem followed by the universal coefficient theorem, we get
$$\pi_2(X)\,\cong\,\pi_2(\wt{X})\,\cong\,H_2(\wt{X};\,\Z)\,\cong\,H^2(\wt{X};\,\Z).$$
    
{\it Proof of (1).}\,\, As was observed in Lemma
\ref{Lem: Minimality of sphere uniformized compact complex surfaces}, $X$ is a minimal surface. We will use the classification
due to Enriques--Kodaira of smooth complex compact minimal surfaces of non-general type. 

{\bf Case 1.} {\it When $\kappa(X)\,=\,-\,\infty$.}

Using the minimality, it follows that $X$ is isomorphic to either $\PP^2$ or a
ruled surface, which, by definition, is topologically a $\PP^1$-bundle over a smooth complete curve $C$. Since $\pi_1(X)$ is infinite, $X$ can't be isomorphic to $\PP^2$. Now consider the ruled surface case. We have the following long exact sequence of homotopy groups
\begin{equation}\label{u1}
\cdots\,\longrightarrow\,\pi_2(\PP^1)\,\longrightarrow\,\pi_2(X)\,\longrightarrow\,\pi_2(C)\,
\longrightarrow\,\pi_1(\PP^1)\,\longrightarrow\,\pi_1(X)\,\longrightarrow\,\pi_1(C)\,\longrightarrow\,(1).
\end{equation}
This implies that $\pi_1(X)\,\cong\,\pi_1(C)$ as $\PP^1$ is simply connected. As $\pi_1(X)$ is infinite, $C$ is of positive genus, and thus $C$ is an Eilenberg-MacLane $K(\pi,\,1)$-space. Thus, $X$ is non-rational. From \eqref{u1} it follows that $\pi_i(X)\,\cong\,\pi_i(\PP^1)\,=\,\pi_i(S^2)$ for all $i\,>\,1$.
By Hurewicz's theorem, $H_2(\wt{X};\,\Z)\,\cong\,\pi_2(\wt{X})\,\cong\,\pi_2(X)\,\cong\,\pi_2(S^2)\,\cong\,\Z$. Since $\wt{X}$ is a holomorphically convex manifold in this case (see Remark \ref{Rem: Shafarevich Conj}),
we have $H_3(\wt{X};\,\Z)\,=\,0$ (using \cite[Proposition 3.7]{GP}), and $\wt{X}$ being a non-compact $4$-manifold it follows that $H_4(\wt{X};\,\Z)\,=\,0$. This implies that $\wt{X}$
is a Moore $M(\Z,\,2)$-space, i.e., $\wt{X}$ is homotopic to $S^2$.

    {\bf Case 2.} {\it When $\kappa(X)\,=\,0$.}

It is known that, in this case, $X$ is uniformized by either a $K3$ surface or $\C^2$. Hence there is no such surface whose universal cover is homotopic to $S^2$.

    {\bf Case 3.} {\it When $\kappa(X)\,=\,1$.}

In this case, $X$ admits an elliptic fibration $f\,:\,X\,\longrightarrow\,D$. By using \cite[Theorem 5.8]{GP}, it 
turns out that $\chi(\mathcal{O}_X)\,=\,e(X)\,=\,0$, since otherwise $\rank \pi_2(X)$ is strictly bigger than $1$ which is not possible 
in our situation. Again by using \cite[Theorem 5.8]{GP}, it is evident that the only possibility for $\pi_2(X)$ being isomorphic to $\Z$ is that $D\,\cong\,\PP^1$ and $f$ has at 
most $3$ singular fibers satisfying the additional condition that if there are exactly three singular fibers then their multiplicities form a 
platonic triple (cf. \cite[Theorem 5.]{GP}).

In the case of elliptic surfaces $X$ it is known that $\wt{X}$ is a holomorphically convex 
manifold (cf. \cite{Gur-Shas}). Therefore, for the smooth projective surfaces admitting above elliptic fibrations, since $$e(X)\,=2\,-\,2b_1(X)\,+\,b_2(X)\,=\,\,0,$$ it turns out that $\pi_1(X)$ is infinite. Once again, it is verified in \cite{GP} that $\pi_2(X)\,\cong\,\Z$. Hence the same argument as in Case 1 proves that $\wt{X}$ is indeed homotopic to $S^2$ for 
such elliptic surfaces.
    
    {\it Proof of (2a).}\,\, Consider the following Cartan--Serre spectral sequence with rational coefficients:
$$E_2^{p,q}\,:=\,H^p(G;\,H^q(\wt{X};\,\Q))\,\Longrightarrow\, H^{p+q}(X;\,\Q).$$
    As $H^i(\wt{X};\,\Q)\,=\,0$, for $i\,\neq\,0,\,2$, we have the following exact sequence of $\Q$-vector spaces
    \begin{alignat}{2}\label{exact seqn}
            0&\,\longrightarrow\,H^2(G;\,\Q)\,\longrightarrow\, H^2(X;\,\Q)\,\longrightarrow\, H^2(\wt{X};\,\Q)^G \nonumber\\ 
            &\, \longrightarrow\, H^3(G;\,\Q)\,\longrightarrow\, H^3(X;\,\Q)\,\longrightarrow\, H^1(G;\,H^2(\wt{X};\,\Q)) \nonumber\\
            &\, \longrightarrow\, H^4(G;\,\Q)\,\longrightarrow\, H^4(X;\,\Q)\,\longrightarrow\, H^2(G;\,H^2(\wt{X};\,\Q))\, \longrightarrow\, H^5(G;\,\Q)\,\longrightarrow\, 0;
    \end{alignat}
    and the isomorphisms
    \begin{equation}\label{eqn: group cohomology shift by 3}
        H^{k+3}(G;\,\Q)\,\cong\,H ^k(G;\,H^2(\wt{X};\,\Q)), \quad \text{for all } k\,\geq\,3
    \end{equation}
    using a similar argument appearing in \cite[Proposition 5.1, Remark 5.2]{BMP}.

    Now using the universal coefficient theorem, followed by Hurewicz's theorem, we have $$H^2(\wt{X};\,\Q)\,\cong \,\Hom(H_2(\wt{X};\,\Q),\, \Q)\,\cong\,\Hom(\Q,\,\Q)\,\cong\,\Q.$$
    Since $H^{4}(G;\,\Q)\,=\,H^{5}(G;\,\Q)\,=\,0$, it follows from the exact sequence \eqref{exact seqn} that 
    \begin{equation}\label{eqn5}
        H^{2}(G;\,\Q)\,\cong\,H^{4}(X;\,\Q)\,\cong\,\Q.
    \end{equation}

    {\bf Claim.} {\it Assume that $\kappa(X)\,=\,2$. Then
\begin{itemize}
\item $H^2(\wt{X};\,\Q)^G\,\cong\,\Q$,

\item the $\Q$-linear map $$H^2(X;\,\Q)\,\longrightarrow\, H^2(\wt{X};\,\Q)^G$$ is surjective.
\end{itemize}}
    
    {\it Proof of Claim.}\,\, Since $H^2(\wt{X};\,\Q)\,\cong\,\Q$, it follows that $H^2(\wt{X};\,\Q)^G$ is either trivial or
it is isomorphic to $\Q$. If possible, let $H^2(\wt{X};\,\Q)^G$ be trivial. Consequently, from \eqref{exact seqn} and
\eqref{eqn5} it follows that $H^2(X;\,\Q)\,\cong\, H^2(G;\,\Q)\,\cong\,\Q$. Hence using the Hodge decomposition, it is evident that $$b_2(X)\,=\,h^{1,1}(X)\,=\,1, \quad \text{and} \quad p_g(X)\,=\,h^{2,0}(X)\,=\,h^{0,2}(X)\,=\,0.$$
    This implies that
    \begin{equation}\label{eqn6}
    e(X)\,=\,2\,-\,2b_1(X)\,+\,b_2(X)\,=\,3\,-\,2b_1(X).
    \end{equation} 
Since $X$ is a minimal surface of general type, then $c_1^2(X)\,>\,0$ (cf. \cite[Chapter VII, Theorem 2.2]{BHPV}. Next, using the Bogomolov-Miyaoka-Yau inequality, it follows that $c_1^2(X)\,\leq\,3c_2(X)$. Thus, 
$$3e(X)\,=\,3c_2(X)\,\geq\,c_1^2(X)\,=\,(K_X)^2\,>\,0,$$ whence $e(X)\,>\,0$. Hodge decomposition yields
that $$b_1(X)\,=\,h^{1,0}(X)\,+\,h^{0,1}(X)\,=\,2h^{1,0}(X)\,=\,2q(X).$$ Therefore, from \eqref{eqn6} it follows immediately that
$b_1(X)\,=\,q(X)\,=\,0$, and hence $c_2(X)\,=\,e(X)\,=\,3$. Now, as we have observed above that $p_g(X)\,=\,q(X)\,=\,0$,
it follows that $$\chi(\OO_X)\,=\,1\,-\,q(X)\,+\,p_g(X)\,=\, 1.$$ Consequently, Noether's formula says that $c_1^2(X)\,=\,12\chi(\OO_X)\,
-\,c_2(X)\,=\,9$. Therefore, it follows that $c_1^2(X)\,=\,3c_2(X)\,=\, 9$ --- a contradiction, since Yau proved that for a general type smooth complex compact surface $M$, if $c_1^2(M)\,=\,3c_2(M)$, then $M$ is uniformized by a disk in $\C^2$ (cf. \cite{Yau}), in particular, then the universal cover $\wt{M}$ of $M$ would become contractible but in our case, $\wt{X}$ is homotopic to $S^2$, a contradiction. This concludes that $H^2(\wt{X};\,\Q)^G\,\cong\,\Q$.

Thus the $\Q$-linear map $H^2(X;\,\Q)\,\longrightarrow\, H^2(\wt{X};\,\Q)^G$ is either surjective or it is
the trivial homomorphism. If the latter happens, once again from \eqref{exact seqn} and \eqref{eqn5} it follows that
$H^2(X;\,\Q)\,\cong\, H^2(G;\,\Q)\,\cong\,\Q$. This leads to the same contradiction as above. Therefore, the above $\Q$-linear map is non-zero which proves the claim.

Consequently, it turns out that the exact sequence \eqref{exact seqn} splits into the following two short exact sequences:
    \begin{align}\label{exact seqn 1}
            0\,\longrightarrow\,H^2(G;\,\Q)\,\longrightarrow\, H^2(X;\,\Q)\,\longrightarrow\, H^2(\wt{X};\,\Q)^G\,\longrightarrow\,0
    \end{align}
    and
    \begin{align}\label{exact seqn 2}
            0\, \longrightarrow\, H^3(G;\,\Q)\,\longrightarrow\, H^3(X;\,\Q)\,\longrightarrow\, H^1(G;\,\Q)\,\longrightarrow\,0.  
    \end{align}
Using the claim and \eqref{eqn5} it follows from \eqref{exact seqn 1} that $H^2(X;\,\Q)\,\cong\, \Q^2$ as $\Q$-vector spaces. Therefore,
the Hodge decomposition gives that $$b_2(X)\,=\,h^{1,1}(X)\,=\,2, \quad \text{and} \quad p_g(X)\,=\,h^{2,0}(X)\,=\,h^{0,2}(X)\,=\,0.$$
    This implies that
    \begin{equation}\label{eqn7}
    e(X)\,=\,2\,-\,2b_1(X)\,+\,b_2(X)\,=\,4\,-\,2b_1(X).
    \end{equation}
    As we observed in the proof of the above claim that $e(X)$ has to be positive, hence once again, it implies that $b_1(X)\,=\,q(X)\,
=\,0$, whence $c_2(X)\,=\,e(X)\,=\,4$. Now, $p_g(X)\,=\,q(X)\,=\,0$ yields
that $$\chi(\OO_X)\,=\,1\,-\,q(X)\,+\,p_g(X)\,=\, 1.$$ Hence as earlier, Noether's formula once again says that
$c_1^2(X)\,=\,12\chi(\OO_X)\,-\,c_2(X)\,=\,8$. Consequently, $$p_g(X)\,=\,q(X)\,=\,0, \quad c_1^2(X)\,=\,2c_2(X)\,=\, 8,$$
so $X$ is a fake quadric.

Since $b_1(X)\,=\,0$, using Poincar{\'e} duality, we have $$H^3(X;\Q) \,\cong\, H_1(X;\,\Q)\,=\,0,$$ and thus from
\eqref{exact seqn 2} it follows immediately that $H^3(G;\,\Q)\,=\,H^1(G;\,\Q)\,=\, 0$. Now using
\eqref{eqn: group cohomology shift by 3} we get that $H^i(G;\,\Q)\,=\, 0$ for all $i\,\neq\,0,\,2$.

    {\it Proof of (2b).}\,\, Suppose that $G\,:=\,\pi_1(X)$ is abelian. Since $X$ is projective, $G$ is finitely presented. Using
the structure theorem of finitely generated abelian groups it follows that, after passing to a suitable finite {\'e}tale covering of $X$,
we have a projective surface $Y$ such that $\pi_1(Y)\,\cong\,\Z^m$ for some positive integer $m$ (note that $m\,>\,0$, which follows
from the fact that $\pi_1(X)$ is infinite). Evidently, $\kappa(Y)\,=\,\kappa(X)\,=\,2$ as $Y \,\longrightarrow\, X$ is a finite
{\'e}tale cover and $\wt{X}$ is the universal cover of $Y$ too. Thus, without loss of generality, we can assume that
$G\,=\,\pi_1(X)\,=\, H_1(X;\,\Z)\,=\,\Z^m$ for some positive integer $m$. Using the Hodge theory, it is easy to observe that $m$
is an even positive integer. We know that for all $i\,>\,0$, $$H^i(G;\,\Z)\,\cong\,H^i(\Z^m;\,\Z)\,\cong\,
H^i\left(K(\Z^m,\,1);\,\Z\right)\,\cong\,H^i(\underbrace{S^1\times\cdots \times S^1}_{m-\text{copies}};\,\Z)\,\cong\,
\Z^{m\choose i}.$$ Since ${m \choose i}\,=\,0$ for $i\,\gg \, 0$, from \eqref{eqn: group cohomology shift by 3} it
follows that $H^3(G;\,\Z)\,=\,0$, and this implies that ${m \choose 3}\,=\,0$, whence $m\,=\,2$ as $m$ is a positive even integer.
Therefore, $$H^1(X;\,\Q)\,=\,H^1(G;\,\Q)\,=\,H^1(\Z^2;\,\Q)\,=\,\Q^2 \quad \text{and} \quad H^2(G;\,\Q)\,=\,H^2(\Z^2;\,\Q)\,=\,\Q.$$ Thus we have the following short exact sequence as a part of the exact sequence in \eqref{exact seqn}:
    \[
            0\,\longrightarrow\,\Q\,\longrightarrow\, H^2(X;\,\Q)\,\longrightarrow\, H^2(\wt{X};\,\Q)^G\,\longrightarrow\,0.
    \]
    Clearly, as $\wt{X}$ is homotopic to $S^2$, so $H^2(\wt{X};\,\Q)^G$ is either trivial or isomorphic to $\Q$. Thus,
in both cases, we have $b_2(X)\,\leq\,2$. We observed above that $b_1(X)\,=\,2$. Hence
$$e(X)\,=\,2\,-\,2b_1\,+\,b_2\,=\,b_2\,-\,2\,\leq\, 0,$$ a contradiction since $e(X)$ has to be positive as we observed in
the proof of the claim in part (\emph{2a}). This contradiction proves that $\pi_1(X)$ is non-abelian.
\end{proof}

\begin{prop}\label{Prop: No compact surface of general type having fundamental group of a compact curve and sphere as universal cover}
    Let $X$ be a smooth projective surface such that its universal cover is homotopic to $S^2$. Then the following statements hold:
    \begin{enumerate}
        \item If $X$ contains a smooth rational curve, then $X$ must be a non-rational ruled surface.
        \item If $G\, :=\,\pi_1(X)$ is a surface group, i.e, the fundamental group of a Riemann surface, then $X$ is either a non-rational ruled surface or an elliptic surface admitting a relatively minimal elliptic fibration $\varphi\,:\, X \,\longrightarrow\, D$ with $D\,\cong \, \PP^1$ such that $\chi(X,\,\mathcal{O}_X)\,=\,0$ and $\varphi$ has at most three multiple fibers with multiplicities forming a platonic triple whenever there are exactly three multiple fibers.
    \end{enumerate}     
\end{prop}

\begin{proof}
First note that, since the universal cover of $X$ is homotopic to $S^2$, the fundamental group of $X$ must be infinite.
    
    {\it Proof of (1).}\,\, Suppose $X$ contains a smooth rational curve $C\,\cong\,\PP^1$. Since $\pi_1(X)$ is infinite, $X$ is a minimal surface and $C^2\,=\, 0$, as observed earlier. Thus, using the proof of \cite[Chapter V, Proposition 4.3]{BHPV}, we conclude that $X$ is a ruled surface admitting a $\PP^1$-bundle structure $f\,:\,X\,\longrightarrow\, D$ on a smooth projective curve $D$ such that $C$ is a full fiber of $f$. Since $\pi_1(X)$ is infinite, $D$ has positive genus. Therefore, $X$ is a non-rational ruled surface.

    {\it Proof of (2).}\,\, Suppose that $G$ is a surface group and assume that $G\,=\,\pi_1(S)$ for a Riemann surface $S$. Since, it is observed earlier that $G\,=\,\pi_1(X)$ is infinite, therefore $S$ is a $K(G,\,1)$-space by the Uniformization Theorem. Hence, 
    \begin{equation}\label{eqn8}
        H^i(G;\,\Z)\,\cong\,H^i(S;\,\Z), \quad \text{for all } i\,\in\,\N.
    \end{equation}
    Suppose, $X$ is of general type. Then the hypothesis of part (a) in the above theorem is satisfied and thus 
    \begin{equation}\label{eqn9}
    H^i(G;\,\Q)\,=\,
    \begin{cases}
        \Q & \text{if } i\,=\,0 \text{ or } 2;\\
        0 & \text{otherwise};
    \end{cases}
    \end{equation}
    and $X$ is a fake quadric surface. If $S$ is a compact Riemann surface, then $g$, the genus of $S$ has to be positive. Thus $H^1(S;\,\Q)\,\cong\, \Q^{2g}$, contradicting the above equations \eqref{eqn8} and \eqref{eqn9}. Now, if $S$ is an open Riemann surface, then a similar contradiction arises since in that case, $H^2(S;\,\Q)\,=\,0$. This implies that $\kappa(X)\,<\,2$ and the rest follows from Theorem \ref{Thm: Characterization of compact surfaces having infinite fundamental group and sphere as universal cover}.
\end{proof}

For the next result, we first recall the definition of \emph{homotopical height} from \cite{BMP}.

Given a group $G$ and a subclass $\mathcal{C}$ of the class of smooth (not necessarily closed) manifolds of positive dimension (e.g. symplectic ($\mathcal{SP}$), K{\"a}hler ($\mathcal{K}$), Stein ($\mathcal{S}$) etc.), it is an old and well-known problem to ﬁnd a manifold $M_G\,\in\, \mathcal{C}$ such that $\pi_1(M_G)\,=\, G$. In \cite{BMP}, the authors refined this concept in the following way:
\begin{definition}
      For a finitely presented group $G$ and a positive integer $r$ ﬁnd $M_{G,r}\,\in\, \mathcal{C}$, if exists, such that $\pi_1(M_{G,r})\,=\, G$ and $\pi_i(M_{G,r})\,=\, 0$ for $1 \,<\, i \,<\, r$. The \emph{$\mathcal{C}$--homotopical height} of $G$ is denoted by ${\rm ht}_{\mathcal{C}}(G)$ and is defined as follows:
      $${\rm ht}_{\mathcal{C}}(G)\,:=\,\max\,\{r\mid \exists \text{ a manifold } M_{G,r} \text{ as above in } \mathcal{C}\}.$$ The ${\rm ht}_{\mathcal{C}}(G)$ is defined to be $-\infty$ if $G$ is not the fundamental group of any manifold in $\mathcal{C}$; and if $\pi_2(M) \,\neq\, 0$ for all $M \,\in\, \mathcal{C}$ with $\pi_1(M)\,=\,G$, then ${\rm ht}_{\mathcal{C}}(G)$ is defined to be $2$.
\end{definition}
Proof of our next result uses the following version of the Generalized Riemann Existence Theorem due to Grauert-Remmert (cf. \cite{Gra-Rem}).

``\emph{Let $f\,:\, Z \,\longrightarrow\, X$ be a proper surjective complex analytic map with finite fibers, where $Z$ and $X$ are irreducible normal complex spaces. If $X$ is an algebraic variety then so is $Z$}''.

\begin{theorem}\label{Thm: Homotopical height of surfaces uniformized by 2-sphere}
    Let $X$ be a smooth projective surface of general type such that its universal cover is a Stein manifold that is homotopic to $S^2$. Let
$G\,:=\,\pi_1(X)$. Then the following statements hold:
    \begin{enumerate}
        \item If ${\rm ht}_{\mathcal{S}}(G)\,>\,2$, then $H^2(G,\,\Z G)\,=\,0$.
        \item If ${\rm ht}_{\mathcal{S}}(G)\,=\,2$ with $M$ being a smooth complex projective surface realizing ${\rm ht}_{\mathcal{S}}(G)$
and the action of $\pi_1(M)$ on $\pi_2(M)$ is the trivial one, then $H^2(G,\,\Z G)\,=\,0$.
    \end{enumerate}
\end{theorem}

\begin{proof}
The proof of the first statement is immediate using Proposition \ref{Prop: No compact surface of general type having fundamental group of a compact curve and sphere as universal cover}(2) followed by \cite[Proposition 6.2]{BMP}.

    Now to prove the second statement, it follows from \cite[Proposition 6.4]{BMP} that it is enough to prove $G$ cannot virtually be a surface group. Suppose $G$ is virtually a surface group and $G$ has a subgroup $G^{\circ}$ such that index $[G\,:\,G^\circ] \,=\,n\, < \, \infty$. Then by the Generalized Riemann Existence Theorem, it is evident that there exists a proper surjective complex analytic map $p\,:\,Z \longrightarrow X$ which is a finite covering map with $$\pi_1(Z)\,\cong\,p_\ast(\pi_1(Z))\,=\,G^\circ.$$ Also, $\deg(p)\,=\,n$ as $[G\,:\,G^\circ] \,=\,n$. By an application of the generalized Riemann existence theorem, it follows that $Z$ is also an algebraic variety since $X$ is projective. Therefore, $Z$ is also a smooth complex projective surface of general type. Again, $Z$ also has the universal cover homotopic to $S^2$ since $X$ has so. Thus, Proposition \ref{Prop: No compact surface of general type having fundamental group of a compact curve and sphere as universal cover}(2) applied to $Z$ yields a similar contradiction as $G^\circ$ is a surface group. This completes the proof.
\end{proof}

\begin{remark}
    In the above theorem, if we assume an affirmative answer to Shafarevich's Conjecture \ref{Conj: SC}, then the universal cover of $X$ being homotopic to $S^2$ automatically implies that its universal cover is a Stein manifold (see the proof of Theorem \ref{Thm: Characterization of quasi-elliptic compact surfaces with infinite fundamental group2}).
\end{remark}

\section*{Acknowledgements} 

We thank the referee for some useful suggestions to improve the exposition in the manuscript.

This work was partly done when both the authors visited Chennai Mathematical Institute (CMI), Chennai, India. We would like to thank CMI for its 
hospitality. The second-named author would like to thank Najmuddin Fakhruddin for the helpful discussions on Shafarevich's conjecture. He is 
thankful to Arnab Roy and Ashutosh Roy Choudhury for their help in the computation of certain spectral sequences; and to Priyankur Chaudhuri for 
suggesting a suitable reference. The first-named author acknowledges the support of a J. C. Bose Fellowship (JBR/2023/000003).

\section*{Statements and Declarations}

There is no conflict of interest regarding this manuscript. No funding was received for it.

\end{document}